\documentclass[12pt]{article}

\usepackage{amsfonts,amsthm,amsmath,amssymb,mathtools,hyperref}
\usepackage[utf8]{inputenc}
\usepackage[english]{babel}

\newtheorem{theorem}{Theorem}
\newtheorem*{theorem*}{Theorem}
\newtheorem{corollary}{Corollary}
\newtheorem{proposition}{Proposition}
\newtheorem{lemma}{Lemma}
\newtheorem{conjecture}{Conjecture}

\theoremstyle{definition}
\newtheorem{definition}{Definition}

\newtheorem{remark}{Remark}

\def \eps {\varepsilon}
\def\ph{\varphi}
\def\pa{\partial}

\title{Singular points  in generic two-parameter families of vector fields on 2-manifold\footnote{This a preprint of the Work accepted for publication in Regular and Chaotic Dynamics Journal, © Pleiades Publishing, Ltd., 2025.; https://pleiades.online.}}
\author{D.A. Filimonov, Yu. S. Ilyashenko}
\date{\today}

\begin{document}

\maketitle

\begin{abstract}
In this paper, we give a full description of all possible  singular points that occur in  generic 2-parameter families of vector fields on compact 2-manifolds. This is a part of a large project aimed to a complete study of global bifurcations in two-parameter families of vector fields on the two-sphere.
\end{abstract}

\begin{flushright}
	\textit{In memory of Leonid Pavlovich Shilnikov.}
\end{flushright}

\section*{Introduction}\label{sec:intro}

An arbitrary dynamical system may exhibit an extremely difficult behavior but such systems are ``rare''. The exact sense of rareness (or more precisely its opposite) is a notion of genericity of a system. The same question may be asked for bifurcations or more generally for finite-parameter families of dynamical systems. In fact, the codimension of a bifurcation is usually meant to be the the lowest number of parameters for which such a bifurcation (and hence a corresponding family) is generic. In our work we deal with finite-parameter families of vector fields on surfaces.
\subsection{Main results}

In this paper, we give a full description of all possible combinations of singular points that occur in generic 2-parameter families of vector fields on compact 2-manifolds. This is a part of a large project aimed to a complete study of global bifurcations in two-parameter families of vector fields on the two-sphere. A first problem to solve is to classify all the degeneracies that may occur in such families. An analogous problem for one-parameter families was solved by Sotomayor in a fundamental paper~\cite{Sotomayor1974}. In particular, he solved an easier problem analogous to one that we consider by a perturbation techniques. These perturbations require many cumbersome calculations. Our main tool is the multijet transversality theorem (published almost simultaneously with the Sotomayor's paper); it allows us to omit almost any calculations.

In addition, for generic  $k$-parameter families we prove the finiteness of the sums of multiplicities of all singular points in the analytical case.

Informally, our main result is the following.

\begin{theorem}\label{thm:sing-residual}
	Vector fields that occur in generic $C^5$-smooth two-parameter families of vector fields on a 2-manifold have  a finite number of singular points. At most two of them are non-hyperbolic and belong to one of the following classes:
	
	an Andronov -- Hopf point with the non-zero first or second focus value (Lyapunov value);
	
	a saddle-node of multiplicity no more than three;
	
	a cuspidal (i.e. Bogdanov - Takens) point with no other degeneracies.
\end{theorem}

A more detailed statement is given in Section \ref{sec:twopar}

\begin{remark}
	Aforementioned Andronov-Hopf points are usually called weak foci but we would like to keep our naming in order to emphasize the relation to a bifurcation with the same name. The same applies to saddle-nodes of higher multiplicities which may be topologically equivalent to a saddle or a node. A cuspidal point here is a degenerate singular point with nilpotent linear part but only those which occurs in Bogdanov-Takens bifurcation.
\end{remark}

Another result is the following

\begin{theorem}\label{thm:k_non_hyp_points}
	Vector fields in generic $k$-parameter families of $C^2$ vector fields on $M^2$ do not have more than $k$ non-hyperbolic singular points.
\end{theorem}
A more detailed statement is given in Section \ref{sec:kpar}.

\begin{remark}
  Wording ``Vector fields in generic  families have this and this property'' means: ``The set of families whose vector fields have this and this property is generic.''
\end{remark}

\subsection{Criticism}
An expert may say that these results  are trivial due to the following general principle: generic $k$-parameter families cannot display bifurcations of codimension more than $k$.

A few years ago the second author would agree with this opinion. He have used the principle quoted above many times when he worked on the books~\cite{AAISh1986} and~\cite{IlyashenkoLi1999}. But recently the first author attracted his attention to the fact that this principle is not a theorem. Indeed, what is its proof?

\emph{The vector fields with the degeneracies of codimension $k+1$ form a manifold of codimension $k+1$ in an appropriate ambient  functional space. By the transversality theorem, generic $k$-parameter family has no intersection with such a manifold.}

What transversality theorem is meant here? The classical one deals with the ambient space of finite dimension. But in our case the ambient space is a Banach space of infinite dimension. The transversality theorem for this case is proved in~\cite{AbrahamRobbin1967}. In this theorem a manifold to which the family should be transversal is a (smooth) Banach submanifold. But even the vector fields with a degenerate singular point form a Banach submanifold with singularities. The Abraham transversality theorem cannot be directly applied to such a ``submanifold''.

\subsection{Revision}

It seems to be a difficult problem to transform the heuristic principle above to a rigorous theorem. Vector field with degeneracies usually form not smooth but stratified submanifolds (this should be proved separately for any particular degeneracy), and the transversality theorem is not stated for this case, no matter how easy it will be. So the statements like theorems 1 and 2 should be proved without the reference to the heuristic principle stated above.

Another problem is that such Banach submanifolds of vector fields even for degeneracies of codimension 1 are not that easy to construct and are sometimes badly immersed (and not embedded) into ambient space. 50 years ago Sotomayor published a large paper [4] where he proved the following statement.

\begin{theorem}
	In generic one-parameter families of vector fields on a compact 2-manifold only the degeneracies of six types may occur. There is a full description of these types.
\end{theorem}

This result easily follows from the principle stated above. But Sotomayor does not use it. What is more important, this principle is not a theorem. Hence, the large and technical paper~\cite{Sotomayor1974} may not be drastically simplified even now. Even the use of aforementioned theorem by Abraham and Robbin needs a lot of work of constructing Banach manifolds for all degeneracies of all codimensions higher than one and requires an analysis of their properties.

The general goal of the authors is to prove an analogous statement for two-parameter families.

\begin{conjecture}
	In generic two-parameter families of vector fields on a surface only the following degeneracies may occur:
	
	-- one or two simultaneous degeneracies from the list in Sotomayor's theorem;
	
	-- exactly one degenenacy from a new finite list.
\end{conjecture}
This will be a rigorous result, and not a consequence of the heuristic principle that was neither proved nor even stated as a theorem.

The present paper is a first step in this direction.

\begin{remark}
	Since our goal is to use in future the infinite dimensional version of transversality theorem fo Banach spaces it forces us to chose a finite-smoothness class of vector fields since $C^\infty$ ones does not form a Banach space. Nevertheless, all the result of the present work are valid for $C^\infty$ and $C^\omega$ families of vector fields.
\end{remark}

\section{Preliminaries} \label{sec:prel}

Two main theorems \ref{thm:sing-residual} and \ref{thm:k_non_hyp_points} of the present paper deal with finite-parameter families of $C^q$ vector fields on compact 2-manifolds $M^2$. The dependence on the parameters is also $C^q$.

\begin{definition}
A base $B^k$ of a family is an open set in $\mathbb R^k$ homeomorphic to a ball and containing
$0: B^k = (\mathbb R^k,0).$ By $Vect^q(M^2)$ we denote the set of all $C^q$ vector fields on $M^2$.
A $C^q$ family $V$ of vector fields on $M^2$ is a $C^q$ map $B^k \to Vect^q(M^2), \eps \mapsto v_\eps$.
\end{definition}

Formally speaking, a vector field is not a map between manifolds. Vector fields on a manifold are sections of a tangent bundle over $M^2$. In order to use any form of the transversality theorem in finite dimensional versions, we have to pass to maps between manifolds. Consider a finite atlas on $M^2$. In any chart $(U_\alpha, \varphi_\alpha)$ a vector field $v$ may be seen as a map:
\begin{equation}
	v\big|_{\varphi_\alpha(U_\alpha)} : D_\alpha \to \mathbb{R}^2
\end{equation}
where $D_\alpha=\varphi_\alpha(U_\alpha)$ is a disc in $\mathbb{R}^2$. Thus any $k$-parameter family of vector fields restricted to an image of this chart is simply a map $B^k\times D_\alpha\to \mathbb{R}^2$.

\subsection{Genericity and transversality}

Let us recall some well known definitions.

There are several ways to formalize the notion of genericity (topological, metrical, several ways to define metrical genericity for infinite dimensional spaces). For the present work we use genericity in a topological sense.

\begin{definition}
	The set is generic (in a topological sense) if it contains a residual subset, i.e. a subset equal to a countable intersection of open dense sets.
\end{definition}

One of the ways to obtain a generic set is the use of Thom's transversality theorem.

\begin{definition}
	A map $f:A\to P$ is \emph{transversal} to a submanifold $W\subset P$ if for every $a \in A$ either $w=f(a)$ is not in $W$, or $w\in W$, and $Df(T_a A) + T_w W = T_w P$. In particular if $\dim A < \mathrm{codim}\, W$ the transversality implies non-intersection of $f(A)$ and $W$.
\end{definition}

The most simple transversality theorem is the following.

\begin{theorem}
	Let $A$ and $P$ be two smooth manifolds where $n=\dim A$ and $\dim P < \infty$, $W \subset P$ be a smooth submanifold  of $P$ of codimension $c$. Then for $q>\max(0,n-c)$, the set of $C^q$-smooth maps
 $f: A \to P$ transversal to $W$ is residual.
\end{theorem}

The space $U$ is called the source, and $P$ the target space.

\subsection{Jet and multijet transversality theorems}

The previous theorem deals with the first derivatives of $f$. But sometimes we need to study the higher derivatives. For this, jets and jet extensions of maps are considered.

\begin{definition}
An $m$-jet of a map $f: A \to P$ at a point $a$ is a set of all $C^m$-smooth maps $g$ defined in some neighborhood  of $a$ in $P$ such that $g(x) - f(x) = o(|x-u|^m)$  as $x \to a$. We denote it by $j^m_a f$.
\end{definition}
If for some representative $f$ of a jet $\{x_1,\ldots,x_n\}$ are local coordinates near $a=\{a_1,\ldots,a_n\}$ on the source manifold $A$ with $\dim A=n$ and $\{y_1,\ldots,y_p\}$ are local coordinates on near $b=f(a)=\{b_1,\ldots,b_p\}$ the target manifold $P$ with $\dim P=p$ than coordinates of the corresponding $m$-jet are
\begin{equation}\label{eq:jet_coordinates}
	j^m_a f = \left<\{a_1,\ldots,a_n\};\{b_1,\ldots,b_p\};\left\{\frac{\partial f}{\partial x}\bigg|_a\right\};\ldots;\left\{\frac{\partial^m f}{\partial x^m}\bigg|_a\right\}\right>
\end{equation}
where $\left\{\frac{\partial^m f}{\partial x^m}\Big|_a\right\}$ denotes all partial derivatives of order $m$ at point $a$. For any $C^m$ map $f: A \to P$ there exists a map called an $m$-jet extension of $f$:

\begin{equation}
	A \to J^m(A,P), \ x \mapsto j^m_x f.
\end{equation}

\begin{theorem} [Thom's jet transversality theorem]
	Let $W$ be a submanifold in the jet space $J^m(A,P)$ of codimension $c=\mathrm{codim}\,W$ and $n = \dim A$. Then the set of maps $f: A \to P$ whose $m-jet$ extension is transversal to $W$ is residual in $C^q(A,P)$ for $q > max(0,n-c)$.
\end{theorem}

Sometimes we need to study the properties of higher derivatives of $f$ at two or more (say, $r$) points simultaneously. For this we need to consider $r$-multijets of $f$ that is, the tuples of jets of $f$ at $r$ different points $a^1,\ldots,a^r$.

\begin{definition}
A tuple $( j^m_{a^1}f,...,j^m_{a^r}f )$ for pairwise distinct $a^i$ is called an $r$-multi-jet of $f$. The space of all such multijets is a subset of $(J^m(A,P))^r$.
\end{definition}

We use the original notation by Mather who first introduced multijets and denote this space by $\prescript{}{r}J^m(A,P)$. We denote by $\pi_i$ a projection of a point in $\prescript{}{r}J^m(A,P)$ to the manifold $A$ in $i$-th component.
\begin{equation}\label{eq:projection}
	\pi_i: ( j^m_{a^1}f,...,j^m_{a^r}f ) \mapsto a^i
\end{equation}

Mather's multijet transversality theorem deals with $C^\infty$ mappings but we need the same theorem in finite smoothness. Here we use this theorem from Francis' article.

\begin{theorem}[Multijet transversality: Francis version  \cite{FrancisBott1972}]\label{thm:multijet_transversality}
	Let $W$ be a submanifold in the space of $r$-multi-$m$-jets $\prescript{}{r}J^m(A,P)$ of codimension $c$ and $q-m>\max (0, rn-c)$, where $n = \dim A$ and $\dim P < \infty$. Then the set of $f\in C^q(A,P)$ for which $\prescript{}{r}j^m
	 f \pitchfork W$ is residual.
\end{theorem}
In our case of families of vector fields the source is $B^k\times M^2$ and the target is $\mathbb R^2$. We denote by $\varepsilon\in I^k$ the parameters of our family and by $x\in M^2$ the phase coordinates. Thus an $r$-multi-$m$-jets of families of vector fields have coordinates
\begin{equation}\label{eq:family_jet_coordinates}
\begin{multlined}
	\left<\left\{\varepsilon^1,x^1,V\left(\varepsilon^1,x^1\right),D^1V\left(\varepsilon^1,x^1\right),\ldots,D^m V\left(\varepsilon^1,x^1\right)\right\},\ldots,\right.\\
	\left.\left\{\varepsilon^r,x^r,V\left(\varepsilon^r,x^r\right),D^1V\left(\varepsilon^r,x^r\right),\ldots,D^m V\left(\varepsilon^r,x^r\right)\right\}\right>
\end{multlined}
\end{equation}
where $D^m V$ denotes the set of all partial derivatives of order $m$ in all variables.

We are interested in vector fields that occur in generic families so we consider the source points  for the same parameter value. In order to make difference between the parameters and the phase variables, we write $v_\varepsilon(x) = V(\varepsilon,x)$. We also need a special notation to write conditions on a jet of a vector field in phase coordinates only. A \textit{truncated} $m$-jet of a family $V = \{v_\varepsilon(x)\}$ is a part of a full jet with the derivatives in the phase coordinates only and without the source coordinate. We denote it by $j^m_*(v_\varepsilon)(x)$ and in the corresponding space its coordinates are
\begin{equation}\label{eq:truncated_jet}
	\left<v_\varepsilon(x),\frac{\partial v_\varepsilon}{\partial x},\ldots,\frac{\partial^m v_\varepsilon}{\partial x^m}\right>
\end{equation}

By abuse of notation we use $j^m_*\left(v_\varepsilon\right)(x)$ also for Taylor expansion of a vector field $v_\varepsilon(x)$ centered at $x$.

\subsection{Normal forms of jets}

In what follows, we characterize the properties of the degenerate singular points by the properties of their normal forms. A general relation between the sets of normal forms and the corresponding vector fields is given by a theorem that makes use of the following definition.

\begin{definition}\label{def:equiv}
	Two $m$-jets of vector fields, $\alpha$ and $\beta$ in $(\mathbb{R}^n, 0)$ are equivalent provided that there exists an $m$-jet $h: (\mathbb{R}^n, 0) \to (\mathbb{R}^n, 0)$ of a diffeomorphism that brings one of these jets into another:
	\begin{equation}
		h_* \alpha = \beta.
	\end{equation}
\end{definition}

We denote the space of $m$-jets of a vector field at zero in $\mathbb{R}^n$  by $J^m(n)$ or simply by $J^m$ where the dimension $n$ is clear.

\begin{theorem}[Takens~\cite{Takens1974} form of Thom's theorem in~\cite{Levine1971}]\label{thm:tak}
	Let $K$ be a semi-algebraic subset of $J^m(n)$. Then the set of all jets of vector fields equivalent to some jet in $K$ is a semi-algebraic subset of $J^m(n)$ again.
\end{theorem}

In what follows we use an improvement of this theorem for some particular cases.

\section{Singular Points in Generic $K$-Parameter Families} \label{sec:kpar}
We start with the exact formulation of the theorem~\ref{thm:k_non_hyp_points}.

\begin{theorem*}[Non-hyperbolic singular points in $k$-parameter families]
	There exists a residual subset of $k$-parameter families of $C^2$ vector fields on two-manifold $M^2$ such that any vector field that occur in a family from this set have no more than $k$ non-hyperbolic singular points.
\end{theorem*}
\begin{proof}[of Theorem~\ref{thm:k_non_hyp_points}]
	Consider a set of vector fields with at least $(k+1)$ non-hyperbolic singular points. We need to prove that there exists a residual subset of $k$-parameter families of vector fields that does not intersect this set. Let a vector field of a family $v_\varepsilon(x)$ for some parameter value $\varepsilon^1$ have $(k+1)$ non-hyperbolic singular points $x^1,\ldots,x^{k+1}$. Its $(k+1)$-multi-1-jet at these points is given by the equations
	\begin{equation}\label{eq:deg_mfd}
		W =
		\begin{cases}
			\varepsilon^1 = \varepsilon^2 = \dots = \varepsilon^{k+1};\\
			v_{\varepsilon^i}\left(x^i\right) = 0 \text{ for } i=1,\ldots,(k+1);\\
			\det \left(\frac{\partial v_{\varepsilon^i}}{\partial x}\left(x^i\right)\right)\cdot \mathrm{tr} \left(\frac{\partial v_{\varepsilon^i}}{\partial x}\left(x^i\right)\right) = 0 \text{ for } i=1,\ldots,(k+1);
		\end{cases}
	\end{equation}
	The first line in~\eqref{eq:deg_mfd} assures that we are writing all conditions for the same parameter value and thus for the same vector field of a family. The set $W$ is semi-algebraic and thus a stratified manifold that corresponds to a vector fields of $V$ for a parameter $\varepsilon^1$ with $(k+1)$ non-hyperbolic singular points. The first line of~\eqref{eq:deg_mfd} contains $k^2$ equations, the second line contains $2(k+1)$ equations and the last line contains $(k+1)$ equations. Thus the codimension of $W$ is $c=k^2 + 3(k+1)$.
	
	Now in order to use the multijet transversality theorem we need to pass to a vector fields defined on one chart of $M^2$. Since $W$ is a submanifold of a jet space (jet bundle over $M^2$) it can be covered by at most countable number of compact sets $K_j$ such that each projection $\pi_i(K_j)$ defined in \eqref{eq:projection}) belongs to one coordinate neighborhood $M^2$ and all these neighborhoods are pairwise disjoint. Thus we can now consider vector fields on a Cartesian product of discs  which are maps to $\mathbb{R}^2$ as explained in section~\ref{sec:intro}.
	
	Let us check the assumptions of the multijet transversality theorem~\ref{thm:multijet_transversality}. In our case $r=k+1$, $n=k+2$, $q=2$, $m=1$ and $rn-c = (k+1)(k+2)-(k^2 + 3(k+1)) = -1$. Since $q-m>0=\max (0, rn-c)$ the set of families of vector fields $V$ such that $\prescript{}{k+1}j^1V\pitchfork W\big|_{K_i}$ is residual. Note that $rn$ is the dimension of a source space for $r$-multi-$m$-jets and since $rn<c=\mathrm{codim}\,W$ the transversality condition means no intersection with $W\big|_{K_i}$. This means that there is a residual set of families of vector fields such that they have no intersection with $W\big|_{K_i}$. Taking a countable intersection for all $K_i$ we obtain a residual set of families that do not intersect the whole $W$.
\end{proof}

\begin{proposition}
	Theorem~\ref{thm:k_non_hyp_points} is true for any compact $N$-dimensional manifold.
\end{proposition}

\begin{proof}
	Indeed, the only difference for higher dimension will be the last line of~\eqref{eq:deg_mfd}: instead of a trace we need another condition for Jacobian matrix to have purely imaginary eigenvalue. Thus we just need a polynomial condition for a matrix to have at least one imaginary eigenvalue.
	
	Let $P(\lambda)$ be a characteristic polynomial of a matrix $A$. Let $P_1(\mu)$ and $P_2(\mu)$ be real and imaginary parts of $P(i\mu)$ assuming $\mu\in\mathbb{R}$. If $P(\lambda)$ has a root $ib,\, b\in\mathbb{R}$ then both real and imaginary parts of $P(ib)$ are zero. Thus $b$ is itself a root of $P_1(\mu)$ and $P_2(\mu)$. Two polynomials have a common root iff their resultant is zero. Let's denote it by $R(A)$. Since it is a polynomial of the coefficients of $P_1$ and $P_2$ it is also a polynomial of elements of the matrix $A$.
	
	Now the last line of~\eqref{eq:deg_mfd} should be the following:
	\begin{equation}
		\det \left(\frac{\partial v_{\varepsilon^i}}{\partial x}\left(x^i\right)\right)\cdot R\left(\frac{\partial v_{\varepsilon^i}}{\partial x}\left(x^i\right)\right) = 0 \text{ for } i=1,\ldots,(k+1);
	\end{equation}
	and all the calculations are the same as in theorem~\ref{thm:k_non_hyp_points} with the change of the dimension from 2 to $N$.
\end{proof}

There is a corollary of theorem~\ref{thm:k_non_hyp_points} concerning multiplicities of singular points for the analytic case. Gabrielov and Khovanski in \cite[corollary after Theorem~7]{GabrielovKhovanskii1998} found an upper bound for the multiplicity of zeroes of a  map $\mathbb{C}^N \to \mathbb{C}^N$ that may occur in a generic analytic $k$-parameter family. For $N = 2$ their formula may be simplified.

\begin{theorem}[Gabrielov and Khovanski~\cite{GabrielovKhovanskii1998}]\label{thm:GKh}
	Let $P_z(x) = (P_{z,1}(x), P_{z,2}(x))$ be a generic family of analytic mappings $\mathbb{C}^2 \to \mathbb{C}^2$ depending on $k$ parameters $z \in \mathbb{C}^k$. Then the multiplicity of $P_z$ at any point $x \in \mathbb{C}^2$ and for any $z$ is less than $\frac{2}{3\sqrt{3}}(k+2)^\frac{3}{2}$.
\end{theorem}

Thus we have a corollary from Theorems \ref{thm:k_non_hyp_points} and \ref{thm:GKh}.

\begin{corollary}\label{cor:analytic_finite_multiplicity}
	Vector fields in generic $k$-parameter families of $C^\omega$ vector fields on $M^2$  have a finite sum of multiplicities of non-hyperbolic singular points. It is not greater than $\frac{2}{3\sqrt{3}}k(k+2)^\frac{3}{2}$.
\end{corollary}

\begin{proof}
	From theorem~\ref{thm:k_non_hyp_points} it follows that vector fields in generic $k$-parameter families of analytic vector fields on two-manifold have no more than $k$ non-hyperbolic singular points. The multiplicity of a singular point for analytic vector field is a local property and in local chart a vector field may be seen as a map $\mathbb{R}^2\to\mathbb{R}^2$. Thus applying theorem~\ref{thm:GKh} we get the statement of the corollary.
\end{proof}

There are several ways to define multiplicity of a singular point for an arbitrary $C^q$ vector field but it seems that corollary~\ref{cor:analytic_finite_multiplicity} is true for $q$ sufficiently large.

\section{Singular Points in Generic Two-Parameter Families} \label{sec:twopar}

\subsection{Classification of linear parts}

Let's introduce the main classes of singular points we deal with. Our classification is close to that given by Dumortier \cite{Dumortier1977}. The space $Hom (\mathbb R^2, \mathbb R^2)$ of linear operators is a disjoint union:
\begin{equation}
Hom (\mathbb R^2, \mathbb R^2) = H^* \bigsqcup AH^* \bigsqcup SN^* \bigsqcup BT^* \bigsqcup \{0\}.
\end{equation}
Here

$H^*$ is the set of all the hyperbolic linear operators of the plane to itself,

$AH^*$  (from Andronov -- Hopf) is the set of all linear operators with two pure imaginary non-zero eigenvalues,

$SN^*$ (from Saddle-Node) is the set of all linear operators with one zero and one non-zero eigenvalue,

$BT^*$ (from Bogdanov -- Takens) is the set set of all non-zero nilpotent linear operators.

Denote by $H, AH, SN, BT,ZL$ (from zero linear) the sets of germs of vector fields on $(\mathbb R^2,0)$ with the linear part from $H^*, AH^*, SN^*, BT^*$ and with zero linear part respectively. We now define special subclasses of these classes that we will need below.

\subsection{Classes $AH_k$}

Germs of class $AH$, have a normal form with a $(2m+1)$-jet equal to
\begin{equation}\label{eq:ah}
j^{2m+1}v = z \left(i \omega + \sum_{k=1}^{m} a_k r^{2k}\right) \frac{\partial}{\partial z}, \omega \in \mathbb R,\ a_k \in \mathbb C,
\end{equation}
where $z=x+iy$ and $r=|z|$.

Denote by $AH^n_0$ ($n$ of normalized) the class of all jets~\eqref{eq:ah} with $\Re a_1 \not = 0$, by $AH^n_k$ the class of all jets~\eqref{eq:ah} with $ \Re a_1 = ... =\Re a_k = 0 \not = \Re a_{k+1}$ for $m \ge k+1$. Obviously, these sets  are semi-algebraic. Denote by $AH_k$ the set of $m$-jets of vector fields equivalent at the origin to some jet from $AH^n_k$.

\begin{remark}
	The class $AH_k$ depends on $m$ as a set of $(2m+1)$-jets. But we skip $m$ in notations to make them simpler.
\end{remark}

The class $AH^n_k$ has codimension $k$ in the set $AH^n_0$. This motivate a statement that germs of vector fields with $k$ first focus values equal to zero constitute a semi-algebraic set of codimension $k$ in $AH_0$, and of codimension $k+1$ in the set of all germs of vector fields at a singular point at the origin. This statement is formalized and proved below.

\subsection{Classes $SN_k$}

Germs of class $SN$ have a normal form with an $m$-jet equal to
\begin{equation}\label{eq:sn}
	j^mv =	\left(  \sum\limits_{k=2}^{m} a_k x^k\right)\frac{\partial}{\partial x} + y\left(\lambda + \sum\limits_{k=1}^{m-1} b_k x^k\right)\frac{\partial}{\partial y}, \ \lambda \not = 0.
\end{equation}

Denote by $SN^n_0$ ($n$ of normalized) the class of all $m$-jets ~\eqref{eq:sn} with $a_2 \not = 0$, by $SN^n_k$ the class of all $m$-jets~\eqref{eq:sn} with $a_2 = ...= a_{k+1}= 0 \not = a_{k+2}$.
Obviously, the sets $SN^n_k, $ are semi-algebraic. Denote by $SN_k$ the set of $m$-jets of vector fields equivalent at the origin to some jet from $SN^n_k$.

By the Takens Theorem~\ref{thm:tak}, the sets $AH_k, SN_k$ are semi-algebraic. But we are interested in their codimension.

\begin{theorem}\label{thm:codim}
	For any $m \ge 2k+3$ the classes $AH_k$ have codimension $k+1$ in $J^m$. For any $m \ge k + 1$ the classes $SN_k$ have codimension $k+1$ in $J^m$.
\end{theorem}

This result was claimed by Takens without proof. It was also used by Dumortier in~\cite{Dumortier1977} for $k=1,2,3,4$ without any justification. The transition from normal forms to all jets seems to be obvious, but should be justified.

Theorem~\ref{thm:codim} for the class $AH$ and $k = 1,2$ may be proved without normal forms because the first two focus values are expressed explicitly through the 3-d and the 5-th jets of the original vector fields, see for instance, \cite{Kusnetsov}. But for the higher focus values this way seems to be hopeless. Theorem~\ref{thm:codim} is proved at the end of the paper.

\subsection{Classes $BT_k, k = 0, 1$}

Germs of class $BT$ have a normal form with a two-jet equal to
\begin{equation}\label{eq:bt}
	y\frac{\partial}{\partial x} + \left(b_{11}x^2+b_{12}xy+b_{22}y^2\right)\frac{\partial}{\partial y}.
\end{equation}

Denote by $BT^n_0$ ($n$ of normalized) the class of all 5-jets  that have the two-jet \eqref{eq:bt} with $b_{11} b_{12}\not = 0$, by $BT^n_1$ the class of all 5-jets that have the two-jet \eqref{eq:sn} with $b_{11} b_{12}= 0$. Obviously, the sets $BT^n_j, \ j = 0,1$ are semi-algebraic. Denote by $BT_j$ the set of 5-jets of vector fields equivalent at the origin to some jet from $SN^n_j$.
It is obvious that
\begin{equation}\label{eqn:codim0}
	\mathrm{codim}\, BT_0 = 2.
\end{equation}

Bogdanov in~\cite{Bogdanov1976} proved that
\begin{equation}\label{eqn:codim1}
	\mathrm{codim}\, BT_1 = 3.
\end{equation}

\subsection{Main theorem}

We are now ready to formulate the precise version of our main result, Theorem~\ref{thm:sing-residual}.

\begin{theorem*}[Singular points in generic two-parameter families]
	There is a finite number of singular points of vector fields in generic two-parameter families of vector fields on compact 2-manifold $M^2$ and they are
	\begin{enumerate}\itemsep=-2pt
		\item either all hyperbolic,
		\item or are all hyperbolic except two which are from the classes $SN_0$ or $AH_0$,
		\item or are all hyperbolic except one from the union:\\
		$W:= SN_0\cup AH_0\cup SN_1\cup AH_1\cup BT $
	\end{enumerate}
	More formally, there exists a residual subset $\mathcal{A}_{sp}$ of $C^q$ ($q\geqslant5$) two parameter families of vector fields that passes only trough vector fields with singular points with the properties mentioned above.
\end{theorem*}

\begin{proof}
	This is a direct consequence of the multijet transversality theorem. First, consider families of vector fields on an open disc $\mathcal{D}^2 \subset \mathbb{R}^2$.

	Consider a set of vector fields that do not satisfy the properties in the conclusion of the theorem. This set contains vector fields with either non-hyperbolic singular point not in $W$ (Case 1), or two non-hyperbolic singular points, one being outside $SN_0\cup AH_0$ (Case 2). By Theorem~\ref{thm:k_non_hyp_points}, a generic two-parameter family contains vector fields with no more than two non-hyperbolic singular points, so there are no other possibilities.

	Case 1. Let $(\varepsilon,x)$ be a singular point of a vector field $v$ of a family $V$ with the truncated 5-jet $j_*^5(v)$ outside $W$. Then either its linear part is zero, and then $j_*^5(v) \in ZL$, or its linear part belongs to $SN \cup AH  \cup BT$. In the latter case, $ j_*^5(v) \in SN_2 \cup AH_2  \cup BT_1 = R_*$.
	This is a semi-algebraic set of codimension 3, and $ZL$ is a semi-algebraic set of codimension 4.
	
	Consider the set $R$ of five jets of vector fields corresponding to Case 1:
	\begin{equation}\label{eq:onej}
		R=
		\begin{cases}
			v_\varepsilon(x)=0\\
			j_*^5(v)\in R_* \cup ZL
		\end{cases}
	\end{equation}
	
	The set $R$ has codimension 5.
	By the transversality theorem~\ref{thm:multijet_transversality} for $r=1$ (this is an original version by Thom), a five-jet extension of a generic map $B^2 \times \mathcal{D}^2 \to \mathbb{R}^2$ is transversal to this set. But the space $B^2 \times \mathcal{D}^2$ has dimension 4. Hence, there exists a residual set of two-parameter families of vector fields $V$, such that the image of the 5-jet extension of $V$ does not intersect $R$.
	
	Case 2 is treated in a similar way. Consider a two-multi-5-jet extension of $V$. This is a map
	\begin{equation}
		\begin{gathered}
			\mathcal V: \left(B^2 \times \mathcal{D}^2 \right)^2 \to \left(J^5\right)^2,\\
			\left(\varepsilon^1, x^1, \varepsilon^2, x^2\right) \mapsto \left(\varepsilon^1, x^1, V\left(\varepsilon^1, x^1\right), D^1V\left(\varepsilon^1, x^1\right), \varepsilon^2, x^2, V\left(\varepsilon^2, x^2\right),D^1V\left(\varepsilon^1, x^1\right)\right).
		\end{gathered}
	\end{equation}

	Let a vector field $v$ of a family $V$ have two non-hyperbolic singular points $x^1$ and $x^2$. Let one of them, say, $x^1$ be outside the  classes $SN_0$ and $AH_0$. Then
	\begin{equation}\label{eqn:twojets}
		\begin{gathered}
			j_*^5\left(v_{\varepsilon^1}\right)\left(x^1\right) \in SN_1\cup AH_1\cup BT \cup ZL,\\
			j_*^5\left(v_{\varepsilon^2}\right)\left(x^2\right) \in SN_0\cup AH_0\cup BT \cup ZL.
		\end{gathered}
	\end{equation}

	Consider the set $R$ of two-multi-5-jets of such vector fields corresponding to Case 2:
	\begin{equation}\label{eqn:twoj}
		R=
		\begin{cases}
			\varepsilon_1=\varepsilon_2\\
			v_{\varepsilon^1}\left(x^1\right)=0\\
			v_{\varepsilon^2}\left(x^1\right)=0\\
			j_*^5\left(v_{\varepsilon^1}\right)\left(x^1\right) \in SN_1\cup AH_1\cup BT \cup ZL\\
			j_*^5\left(v_{\varepsilon^2}\right)\left(x^2\right) \in SN_0\cup AH_0\cup BT \cup ZL.
		\end{cases}
	\end{equation}
	
	The set $SN_1\cup AH_1\cup BT \cup ZL$ has codimension two, and the set  $SN_0\cup AH_0\cup BT \cup ZL$ has codimension one. Thus we conclude that the set of all the two-multi-5-jets \eqref{eqn:twoj} has  codimension 9. But the source space $\left(B^2 \times \mathcal{D}^2\right)^2$ is of dimension 8. Hence, there exists a residual set of two-parameter families of vector fields $V$, such that the image of the 5-jet extension of $V$ does not intersect $R$.
	
	Thus there exists a residual subset of two-parameter families of vector fields on any open disc $\mathcal{D}^2\subset\mathbb{R}^2$ such that any vector field that occur in these families satisfy the properties 1-3 in the conclusion of the theorem.
	
	Now we proceed in the same way as in the proof of the theorem~\ref{thm:k_non_hyp_points}. For a compact manifold $M^2$ the sets defined by~\eqref{eq:onej} and~\eqref{eqn:twoj} are submanifolds in corresponding jet or multijet spaces. Each of them can be covered by a countably many compact sets that have projections into pairwise disjoint coordinate neighborhoods on $M^2$. For any such a compact set we have a desired result considering families of vector fields on discs in these neighborhoods. By taking a countable intersection of residual sets acquired for families of vector fields on discs we obtain a residual subset $\mathcal{A}_{sp}$ for two-parameter families on $M^2$.
	
	Note that hyperbolic points and non-hyperbolic ones that belong to classes $SN_0$, $AH_0$,  $SN_1$, $AH_1$ and $BT$ are all isolated ones (see for example~\cite{Dumortier1977}). Thus since $M^2$ is a compact manifold, any vector field occurring in families form $\mathcal{A}_{sp}$ has only a finite number of singular points.
\end{proof}

\section{Codimension of Degeneracies}

In this section we prove Theorem \ref{thm:codim}

\subsection{Centralizers}

Denote by $J^m$ the set of $m$-jets of vector fields at a singular point zero in $\mathbb R^2$, and by $D^m$ the set of $m$-jets of diffeomorphisms of $\mathbb R^2$ at a fixed point 0.
The set $D^m$ is a Lie group that acts of $J^m$ by the formula:
$$
g(v) = g_*v.
$$
The dimension of an orbit of $D^m$ passing through a jet $v$ equals to:
\begin{equation}
	\dim \mathrm{orb}\, v = \dim D^m - \dim \mathrm{Stab}\, v,
\end{equation}
where $\mathrm{Stab}\,v$ is the stabilizer of $v$ in $D^m$:
\begin{equation}
	\mathrm{Stab}\,v = \{g \in D^m|g_*v = v\}.
\end{equation}
The set $\mathrm{Stab}\,v$ is a Lie group whose Lie algebra $\mathrm{Cent}\, v$ (centralizer of $v$) consists of $m$-jets of vector fields that commute with $v$:
\begin{equation}
	\mathrm{Cent}\,v = \{w\in J^m \big|[v,w] = 0 \text{ in } J^m \}.
\end{equation}

\begin{lemma}\label{lem:cah}
	For any $v \in AH^n_k, m \ge 2k + 3$, different from a linear jet,
	\begin{equation}
		\dim \mathrm{Cent}\, v = 2
	\end{equation}
\end{lemma}

\begin{lemma}\label{lem:csn}
	For any $v \in SN^n_k, m \ge k + 1$, different from a linear jet,
	\begin{equation}
		\dim \mathrm{Cent}\, v = 2
	\end{equation}
\end{lemma}
These ``centralizer lemmas'' are proved below.

\subsection {Proof of the Theorem \ref{thm:codim} modulo the Centralizer lemmas}

By definition, the class $AH_k$ consists of jets whose normal form belongs to $AH_k^n$. Hence,

\begin{equation}
	AH_k = \bigcup_{v \in AH^n_k} \mathrm{orb}\, v,
\end{equation}

\begin{equation}
	\dim AH_k = \dim AH_k^n + \dim D^m - 2
\end{equation}

By Lemma \ref{lem:cah}, the dimension of the orbits of $v \in AH^n_0$ is the same. Hence,
\begin{equation}
	\dim AH - \dim AH_k = \dim AH^n - \dim AH^n_k = k.
\end{equation}
This proves Theorem \ref{thm:codim} for $v \in AH$ modulo Lemma \ref{lem:cah}.

In the same way Theorem \ref{thm:codim}  for $v \in SN$  is proved modulo Lemma \ref{lem:csn}.

\subsection {Proof of Centralizer lemma for the class AH}

Let $v \in AH, w \in J^m$. Let us pass to polar coordinates. Then for $v$ from \eqref{eq:ah}
\begin{equation}\label{eqn:vahn}
	v = f(r)\frac{\pa}{\pa r} + g(r)\frac{\pa}{\pa \ph}
\end{equation}
where
\begin{equation}
	f(r) = r \sum_{j=1}^m \Re a_jr^{2j}, \ g(r) =  \omega + \sum_{j=1}^m \Im a_jr^{2j}, \ \omega \not = 0.
\end{equation}
Let
\begin{equation}\label{eq:w_AH}
	\begin{gathered}
		w = F\frac {\pa}{\pa r} + G\frac {\pa}{\pa \ph} = \sum_{n=-\infty}^{\infty} c_n(r) e^{in\ph} \frac {\pa}{\pa r} + \sum_{n=-\infty}^{\infty} d_n(r) e^{in\ph}\frac {\pa}{\pa \ph},\\
		\overline c_n = c_{-n}, \ \overline  d_n = d_{-n}.
	\end{gathered}
\end{equation}
In coordinates, $v$  and $w$ are vector functions, and
\begin{equation}
	[v,w] = L_v w - L_w v = \left(F_r f + F_\ph g - f'F\right)\frac {\pa}{\pa r}
+ \left(G_r f + G_\ph g - g'G\right)\frac {\pa}{\pa \ph},
\end{equation}
where primes are derivatives in $r$.

Relation $[v,w] = 0$ implies
\begin{equation}
	L_v w = L_w v.
\end{equation}
Substituting $v$ and $w$ from~\eqref{eqn:vahn} and~\eqref{eq:w_AH} and equating the coefficients before $e^{in\ph}$, we get
\begin{equation}\label{eq:cn}
	c_n' = \left( \frac {f'}{f} - in \frac {g}{f}\right) c_n,
\end{equation}

\begin{equation}\label{eq:dn}
	d_n' = \left( -in \frac {g}{f}\right) d_n +  \frac {g'}{f} c_n,
\end{equation}

For $n \neq 0$, equation~\eqref{eq:cn} has an irregular singularity at 0, and has no solutions in the jet space at 0, except for identical zero. So $c_n \equiv 0$ for $n \ne 0$. Then,  for this $n$,
\begin{equation}
	d_n' = \left( -in \frac {g}{f}\right) d_n.
\end{equation}
The same argument implies that $d_n \equiv 0$ for $n \ne 0$.

For $n = 0$, $c_0$ is real, hence $c_0 = \alpha f, \ \alpha \in \mathbb R$. Then
\begin{equation}
	d_n' = \alpha g' \Rightarrow d_n = \alpha g + \beta.
\end{equation}
The vector field
$w = \alpha f(r)\frac{\pa}{\pa r} + \left( \alpha g(r)+ \beta\right) \frac{\pa}{\pa \ph} $
obviously commutes with $v$.

Lemma \ref{lem:cah} for the class AH is proved.


\subsection {Proof of Centralizer lemma for the class SN}

The proof is similar to the previous one.

Let $v \in SN, w \in J^m$. For $v$ from \eqref{eq:sn}
\begin{equation}
	v = f(x)\frac{\pa}{\pa x} + y g(x)\frac{\pa}{\pa y}
\end{equation}
where
\begin{equation}
	f(0) = f'(0) = 0, \ g(0) \not = 0.
\end{equation}
Let
\begin{equation}
	w = F\frac {\pa}{\pa x} + G\frac {\pa}{\pa y} = \sum_{n=0}^{\infty} c_n(x) y^n \frac {\pa}{\pa x} + \sum_{n=1}^{\infty} d_n(x) y^n \frac {\pa}{\pa y}.
\end{equation}
In coordinates, $v$  and $w$ are vector functions, and
\begin{equation}\label{eq:w_SN}
	[v,w] = L_v w - L_w v = \left(F_x f + F_y g - f'F\right)\frac {\pa}{\pa x} + \left(G_x f + G_y g - g'G\right)\frac {\pa}{\pa y},
\end{equation}
where primes are derivatives in $x$.

Relation $[v,w] = 0$ implies
\begin{equation}
	L_v w = L_w v.
\end{equation}
Substituting $w$ from~\eqref{eq:w_SN} and equating the coefficients before $y^n$, we get
\begin{equation}\label{eqn:cns}
	c_n' = \left( \frac {f'}{f} - n \frac {g}{f}\right) c_n,
\end{equation}

\begin{equation}\label{eqn:dns}
	d_n' =  (1-n) \frac {g}{f} d_n +  \frac {g'}{f} c_{n-1},
\end{equation}

For $n \neq 0$, equation~\eqref{eqn:cns} has an irregular singularity at 0, and has no solutions in the jet space at 0, except for identical zero. So $c_n \equiv 0$ for $n \ne 0$.

The same argument implies that $d_n \equiv 0$ for $n \ne 1$.

For $n = 0$,  $c_0 = \alpha f, \ \alpha \in \mathbb R$.

For $n = 1, d_n$ satisfies the equation
$$
 d_1' = \frac {g'}{f} c_0 = \alpha g'.
$$
Hence, $d_1 = \alpha g + \beta.$

The vector field
$w = \alpha f(x)\frac{\pa}{\pa x} + ( \alpha g(x)+ \beta) \frac{\pa}{\pa y} $
obviously commutes with $v$.

Lemma \ref{lem:csn} for the class SN is proved.


\begin{thebibliography}{99}

\bibitem{AbrahamRobbin1967}
R. Abraham and J. Robbin, \textit{Transversal Mappings and Flows}. New York, Amsterdam, 1967.

\bibitem{AAISh1986}
Afraimovich, V., Arnold, V., Ilyashenko, Yu., Shilnikov, L., \textit{Dynamical systems V. Bifurcation theory and catastrophe theory}. Springer 1994; translation from \textit{Itogi Nauki Tekh., Ser. Sovrem. Probl. Mat., Fundam. Napravleniya 5}, Nauka: Moscow, 1986.						

\bibitem{Bogdanov1976}
Bogdanov R.\,I., The versal deformation of a singular point of a vector field on the plane in the case of zero eigenvalues, in \textit{Proceedings of Petrovskii Seminar}, Moscow: Moscow State University, 1976, vol.\,2, pp.\,37--65 (Russian).

\bibitem{Dumortier1977}
Dumortier F., Singularities of vector fields on the plane, \textit{Journal of Differential Equations}, vol.\,23, pp.\,53--106, 1977.

\bibitem{FrancisBott1972}
Francis G.\,K., Generic homotopies of immersions, \textit{Indiana University Mathematics Journal}, vol.\,21, no.\,12, pp.\,1101--1112, 1972.

\bibitem{GabrielovKhovanskii1998}
Gabrielov A. and Khovanskii A., Multiplicity of a noetherian intersection. in  \textit{Geometry of differential equations} Khovanskii A., Varchenko A., and Vassiliev V. (Eds.),, Providence, R.I: American Mathematical Society, 1998, pp.\,119--130.

\bibitem{IlyashenkoLi1999}
Yu. Ilyashenko Yu., Li W., \textit{Nonlocal bifurcations},  AMS: Providence, RI, 1999.

\bibitem{Kusnetsov} Kuznetsov Yu.\,A., \textit{Elements of applied bifurcation theory}, Springer: Cham, Switzerland, 2023.

\bibitem{Levine1971}
Levine H.\,I., Singularities of differentiable mappings, notes of lectures of R. Thom in Bonn (1960), in {Proceedings of Liverpool Singularities --- Symposium I}, Wall C.\,T.\,C. (Ed.), Berlin, Heidelberg, New-York: Springer-Verlag, 1971, pp.\,1--89.

\bibitem{Mather1969}
Mather J.\,N., Stability of $C^\infty$ mappings: II. Infinitesimal stability implies stability, \textit{The Annals of Mathematics}, 1969, vol.\,89, no.\,2, pp.\,254--291.

\bibitem{Sotomayor1974}
Sotomayor J., Generic one-parameter families of vector fields on two-dimensional manifolds, \textit{Publications Math\'ematiques de l'IH\'ES}, 1974, vol.\,43, pp.\,5--46.

\bibitem{Takens1974}
Takens F., Singularities of vector fields, \textit{Publications Math\'ematiques de l'IH\'ES}, 1974, vol.\,43, pp.\,47--100.

\end{thebibliography}
\end{document}